\documentclass[12pt, reqno]{amsart}
\usepackage{amsmath, amstext, amsbsy, amssymb, amscd}
\usepackage{amsxtra}
\usepackage{amscd}
\usepackage{amsthm}
\usepackage{amsfonts}
\usepackage{eucal}
\usepackage{color}
\usepackage[all]{xy}
\usepackage[CJKbookmarks=true]{hyperref}

\newtheorem{theorem}{Theorem}[section]
\newtheorem{lemma}[theorem]{Lemma}
\newtheorem{prop}[theorem]{Proposition}
\newtheorem{corollary}[theorem]{Corollary}
\theoremstyle{definition}

\newtheorem{defn}[theorem]{Definition}

\newtheorem{remark}[theorem]{Remark}

\numberwithin{equation}{section}

\def\p{\text{p}}

\def\End{\text{End}}

\def\GL{{\text{GL}}}
\def\SP{\text{SP}}
\def\SO{\text{SO}}
\def\O{\text{O}}

\def\Lie{\text{Lie}}

\makeatletter

\newcommand{\Rmnum}[1]{\expandafter\@slowromancap\romannumeral #1@}
\makeatother

{\vskip-\lastskip\medskip
  \noindent
  {\em #1.}\enspace
  }%
{\qed\par\medskip
  }

\begin{document}

\title[Vust's Theorem and higher level Schur-Weyl duality]{Vust's Theorem and higher level Schur-Weyl duality for types $B$, $C$ and $D$}

\author{Li Luo and Husileng Xiao}

\address{Department of Mathematics, East China Normal University, Shanghai 200241, China.}\email{lluo@math.ecnu.edu.cn, hsl1523@163.com.}

\begin{abstract}
Let $G$ be a complex linear algebraic group, $\mathfrak{g}=\Lie(G)$ its Lie algebra and $e\in\mathfrak{g}$ a nilpotent element. Vust's theorem says that in case of $G=\GL(V)$, the algebra $\mbox{End}_{G_e}(V^{\otimes d})$, where $G_e\subset G$ is the stabilizer of $e$ under the adjoint action, is generated by the image of the natural action of $d$-th symmetric group $\mathfrak{S}_d$ and the linear maps $\{1^{\otimes (i-1)}\otimes e\otimes1^{\otimes (d-i)}|i=1,\ldots,d\}$. In this paper, we generalize this theorem to $G=\O(V)$ and $\SP(V)$ for nilpotent element $e$ with $\overline{G\cdot e}$ being normal. As an application, we study the higher Schur-Weyl duality in the sense of \cite{BK2} for types $B$, $C$ and $D$, which establishes a relationship between $W$-algebras and degenerate affine braid algebras.
\end{abstract}

\maketitle
\section{Introduction}
The classical Schur-Weyl duality, named after two pioneers of representation theory, shows a double centralizer property between the general linear group $\GL(V)$ and the symmetric group $\mathfrak{S}_d$.
Precisely, the $d$-fold tensor space $V^{\otimes d}$ admits a $(\mathbb{C}\GL(V),\mathbb{C}\mathfrak{S}_d)$-bimodule structure, where $\mathfrak{S}_d$ acts by permuting the tensor positions and $\GL(V)$ acts naturally in each tensor position. If we name the representations as follows
$$\mathbb{C}\GL(V)\stackrel{\varphi}{\curvearrowright} V^{\otimes d}\stackrel{\sigma}{\curvearrowleft} \mathbb{C}\mathfrak{S}_d$$
then
\begin{eqnarray*}
\End_{\GL(V)}(V^{\otimes d})&=& \sigma(\mathbb{C}\mathfrak{S}_d);\\
\varphi(\mathbb{C}\GL(V))&=&\End_{\mathfrak{S}_d}(V^{\otimes d}).
\end{eqnarray*}
Differentiating the action of $\GL(V)$, we obtain an action (denoted by $\phi$) of its Lie algebra $\mathfrak{gl}(V)$ on $V^{\otimes d}$. The following is an alternative statement of Schur-Weyl duality:
 \begin{eqnarray*}
\End_{\mathfrak{gl}(V)}(V^{\otimes d})&=& \sigma(\mathbb{C}\mathfrak{S}_d);\\
\phi(\textbf{U}(\mathfrak{gl}(V)))&=&\End_{\mathfrak{S}_d}(V^{\otimes d}).
\end{eqnarray*}

Nowadays there are varieties of generalizations for this duality. Its quantum analogue was studied by Jimbo \cite{Ji} where symmetric groups and universal enveloping algebras are replaced by Iwahori-Hecke algebras and quantum groups, respectively. The super version was achieved by Sergeev \cite{S}, who established a double centralized property between the Lie superalgebra $\mathfrak{gl}_{m|n}$ and $\mathfrak{S}_d$.

For other classical algebraic groups $G=\O(V)$ or $\SP(V)$,  Brauer \cite{B} introduced a series of algebras (now named Brauer algebras) and showed that $G$ and Brauer algebras form an analogue of Schur-Weyl duality for types $B, C$ and $D$.

Moreover, Vust considered another interesting generalization of Schur-Weyl duality.  Let $G=\GL(V)$, $\mathfrak{g}=\mathfrak{gl}(V)$ its Lie algebra and $e\in\mathfrak{g}$ a nilpotent element. Denote the centralizer of $e$ in $G$ by
$$G_e:=\{g\in G|g^{-1}eg=e\}.$$
For any $1\leq i\leq d$, write
\begin{equation}\label{def:ei}
e^{(i)}:=1^{\otimes (i-1)}\otimes e\otimes1^{\otimes (d-i)}\in\End(V^{\otimes d}).
\end{equation}

Denote by $\mathfrak{S}_d[e]$ the subalgebra of $\End(V^{\otimes d})$ generated by $\sigma(\mathfrak{S}_d)\cup\{e^{(i)}|1\leq i\leq d\}$. Vust's Theorem (c.f. \cite{KP1}) says that
\begin{equation}\label{Vust:typeA}
\End_{G_{e}}(V^{\otimes d})=\mathfrak{S}_{d}[e].
\end{equation}
Its arbitrary characteristic version was proved by Donkin in \cite{D}.

Denoted by
$$\mathfrak{g}_e:=\mbox{Lie}(G_e)=\{X\in\mathfrak{g}|[X,e]=0\}.$$
Based on Vust's Theorem, Brundan and Kleshchev \cite{BK2} established a duality between $\mathfrak{g}_e$ and $\mathfrak{S}_{d}[e]$.
Then they developed its filtered deformation, which is called \emph{higher level Schur-Weyl duality}. This duality shows a double centralizer property between the $W$-algebras of type $A$ and the cyclotomic Hecke algebras.

In this paper, we will investigate the Vust's theorem for types $B$, $C$ and $D$, and then study the higher level Schur-Weyl duality for these types. The main results of this present paper are Theorem \ref{mainresult} and \ref{higr shur w dual}. Throughout this paper, the base field is the complex number field $\mathbb{C}$ (any algebraically closed field of characteristic zero is fine, too).

We would like to point out here that there is also another kind of Schur-Weyl duality different from Brauer's setting. Note that the symmetric group $\mathfrak{S}_d$ is the Weyl group of type $A$. It is natural to consider the duality when $\mathfrak{S}_d$ is replaced by Weyl groups of other types.
We refer to Green's work \cite{Gre} about this issue. Furthermore, its quantum analogue, developed by Bao and Wang \cite{BW}, can be used to give a new approach to Kazhdan-Lusztig Theory. Chen, Guay and Ma's work \cite{CGM} about the duality between Yangians and affine Hecke algebras is also with this taste.
We will study the higher level Schur-Weyl duality for this different setting in a subsequent paper, which may provide a relationship between W-algebras and Yangians for type $B/C$.

The paper is organized as follows. Section 2 is devoted to generalizing Vust's Theorem. In Section 3 we study the higher level Schur-Weyl duality for types $B,C$ and $D$.

\section{Vust's Theorem for types $B,C$ and $D$}\label{proof of main result }

This section is mainly devoted obtaining Vust's Theorem for types $B, C$ and $D$ (i.e. Theorem \ref{mainresult}).

\subsection{Trace function}
Let $G=\O(V)$ or $\SP(V)$, and $\langle,\rangle$ be the defining quadratic form on $V$ for $G$. For each $X\in\End(V)$, denote by $X^{\iota}\in\End(V)$ the unique element satisfying
$\langle Xv,u\rangle=\langle v,X^{\iota}u\rangle$ for any $u,v\in V$. In particular, $(X^\iota)^\iota=X$. Furthermore,
\begin{equation}\label{iotalie}
\mbox{$X\in\mathfrak{g}=Lie(G)$ iff $X^{\iota}=-X$.}
\end{equation}

There is a bijection $\theta: V^{\otimes 2}\rightarrow\End(V)$ determined by
\begin{equation}\label{isotheta}
\theta(u\otimes w)(v):=\langle w,v\rangle u,\quad (\forall u,w,v\in V).
\end{equation}
It is clear that $$\mbox{Trace}(\theta(u\otimes w))=\langle w,u\rangle$$ and hence
\begin{equation}\label{iotatrace}
\mbox{Trace}(\theta(Xu\otimes w))=-\mbox{Trace}(\theta(u\otimes X^\iota w)).
\end{equation}

\begin{lemma}\label{Tr of product}
(1). If $X=\theta(u\otimes w)\in \End(V)$, then $X^{\iota}=\theta(w\otimes u).$\\
(2). Let $X_i=\theta(u_i\otimes w_i)$ where $u_i,w_i\in V$ for $i=1,2,\ldots, k$. Then
$$X_1 X_2 \cdots X_{k}=\langle w_{1},u_{2}\rangle\langle w_{2},u_{3}\rangle \cdots \langle w_{k-1},u_{k}\rangle \theta(u_{1}\otimes w_{k}),$$ and hence
$$\mbox{Trace}(X_1 X_2 \cdots X_k)=\langle w_{1},u_{2}\rangle\langle w_{2},u_{3}\rangle \cdots \langle w_{k},u_{1}\rangle.$$
\end{lemma}
\begin{proof}
The first statement follows from the following computation:
$$\langle\theta(u\otimes w)(v_1),v_2\rangle=\langle w,v_1\rangle\langle u,v_2\rangle=\langle v_1, \theta(w,u) v_2\rangle, \quad \forall v_1,v_2\in V.$$

For the second statement, we can show that for any $v\in V$,
\begin{eqnarray*}
X_1 X_2 \cdots X_{k}(v)&=&\langle w_k,v\rangle X_1 X_2\cdots X_{k-1}(u_k)\\
&=&\langle w_k,v\rangle\langle w_{k-1},u_k\rangle X_1 X_2\cdots X_{k-2}(u_{k-1})\\
&=&\cdots\cdots\\
&=&\langle w_k,v\rangle\langle w_{k-1},u_k\rangle\cdots\langle w_1, u_2\rangle(u_1)\\
&=&\langle w_{1},u_{2}\rangle\langle w_{2},u_{3}\rangle \cdots \langle w_{k-1},u_{k}\rangle \theta(u_{1}\otimes w_{k})(v).
\end{eqnarray*}
\end{proof}

\subsection{$G$-invariant ring}
Let $\mathbb{C}[\End(V)^{\oplus d}]$ be the polynomial function ring of $\End(V)^{\oplus d}$. The conjugation action of $G$ on $\End(V)$ induces an action of $G$ on $\mathbb{C}[\End(V)^{\oplus d}]$.
Write \begin{align*}
&\mathbb{C}[\End(V)^{\oplus d }]^{G}:=\\
&\left\{f\in \mathbb{C}[\End(V)^{\oplus d }]\middle|\begin{array}{c}
f(X_1,X_2,\ldots,X_d)=\\f(g^{-1}X_1g,g^{-1}X_2g,\ldots,g^{-1}X_d g),
\end{array} \begin{array}{l}
\forall g\in G \mbox{ and }\\ X_1,X_2,\ldots, X_d\in \End(V)
\end{array} \right\}
\end{align*}
to be the invariant ring for the action of $G$ on $\mathbb{C}[\End(V)^{\oplus d}]$.

\begin{theorem}[c.f. Theorem 7.1 in \cite{Pro}]\label{mith}
For $G=\O(V)$ or $\SP(V)$, the invariant ring
$\mathbb{C}[\End(V)^{\oplus d }]^{G}$ is generated by functions $f$ in form of
$$f(X_1,X_2,\ldots,X_d)=\mbox{Trace}(U_{i_{1}}\cdots U_{i_{k}}),$$
where $U_{j}=X_{j}$ or $X_{j}^{\iota}$, $k\in \mathbb{N}$ and $1\leq i_1,\ldots,i_k\leq d$.
\end{theorem}

\subsection{Action of Brauer algebra on $V^{\otimes d}$}
The original definition of Brauer algebras involves $d$-diagrams with $2d$ vertices and $d$ edges. Since it would occupy too much space but will never be used in this paper, we refer to \cite{B} (also c.f. \cite{Gro}) for this definition. Instead, we describe the image of Brauer algebra in $\End(V^{\otimes d})$ in the following.

Take a basis $\{ v_{p} \mid 1 \leq p \leq n\}$ of $V$, and let
$\{ v^{p} \mid 1 \leq p \leq n\}$ be the dual basis (i.e. $\langle v_{p},v^{q}\rangle=\delta_{ij}$).
Define $\gamma_{ij}\in \mbox{End}(V^{\otimes d}) (i\neq j)$ by
$$\gamma_{ij}(u)=\langle u_{i},u_{j}\rangle\sum^{n}_{p=1}u_{1}\otimes \dots \otimes v_{p}\otimes \dots \otimes v^{p}\otimes \dots \otimes u_{d}$$
for any $u=u_{1}\otimes \dots \otimes u_{d} \in V^{\otimes d}$. It is known that $\gamma_{ij}$ is independent on the choice of $\{v_{p} \mid 1 \leq p \leq n\}$.

Let $B_d$ be the subalgebra of $\End(V^{\otimes d})$ generated by $\{\gamma_{ij}| 1\leq i\neq j\leq n\}$ and $\sigma(\mathfrak{S}_{d})$. It is known (c.f. Proposition 10.1.3 in \cite{GW}) that $B_d$ is the image of Brauer algebra in $\End(V^{\otimes d})$.

\subsection{Some technical lemmas}
For any $\mathbf{l}=(l_{1},\dots l_{d})\in\mathbb{Z}_{\geq 0}^{d}$ and $X\in\mathfrak{g}$, set
$$X(\mathbf{l}):=X^{l_{1}}\otimes\cdots\otimes X^{l_{d}}\in \End(V^{\otimes d}).$$

\begin{lemma}\label{change X to X'} Take $Y=\theta(u_{1}\otimes w_{1}) \otimes \cdots \otimes \theta(u_{d}\otimes w_{d})\in\End(V^{\otimes d})$ where $u_i,w_i \in V, (i=1,2,3,\ldots,d)$. For any $b \in B_{d}$, $\mathbf{l}=(l_{1},\dots l_{d})\in\mathbb{Z}_{\geq 0}^{d}$ and $X\in\mathfrak{g}$, we have
$$\mbox{Trace}(X(\mathbf{l})\circ b \circ Y)= (-1)^{\sum_{i=1}^d l_{i}} \mbox{Trace}( b \circ Y')$$
where
$$Y'=\theta(u_{1}\otimes X^{l_1}w_{1})\otimes\cdots\otimes\theta(u_{d}\otimes X^{l_d}w_{d}).$$
\end{lemma}

\begin{proof}
For any $s \in \sigma(\mathfrak{S}_{d})\subset B_d$ and $v_1,\ldots,v_d\in V$,
\begin{eqnarray*}
s\circ Y(v_1\otimes \cdots\otimes v_d)&=&s(\langle w_1,v_1\rangle u_1\otimes \cdots\otimes\langle w_d,v_d\rangle u_d)\\
&=&\langle w_1,v_1\rangle u_{s(1)}\otimes \cdots\otimes\langle w_d,v_d\rangle u_{s(d)}\\
&=&(\theta(u_{s(1)}\otimes w_{1})\otimes\cdots\otimes\theta(u_{s(d)}\otimes w_{d}))(v_1\otimes \cdots\otimes v_d).
\end{eqnarray*}
That is,
$$s\circ Y=\theta(u_{s(1)}\otimes w_{1})\otimes\cdots\otimes\theta(u_{s(d)}\otimes w_{d}).$$

Similarly, for $\gamma_{ij} \in B_d $ we have
$$\gamma_{ij}\circ Y=\langle u_{i},u_{j}\rangle\sum_{p=1}^{n}\theta(u_{1}\otimes w_{1})\otimes\cdots\otimes\theta(v_{p}\otimes w_{i})\otimes\cdots\otimes\theta(v^{p}\otimes w_{j}) \otimes\cdots\otimes\theta(u_{d}\otimes w_{d}).$$

Hence we can assume that for any $b\in B_d$,
\begin{equation*}
 b\circ Y=\sum  \theta(\Box_1 \otimes w_{1})\otimes\cdots\otimes \theta(\Box_d \otimes w_{d}).
\end{equation*}

Therefore by \eqref{iotalie} and \eqref{iotatrace}, we have
\begin{align*}
\mbox{Trace}(X(\mathbf{l})\circ b \circ Y)
&=\mbox{Trace}(\sum \theta((X^{l_{1}}\Box_1) \otimes w_{1})\otimes \cdots \otimes\theta((X^{l_{d}}\Box_d) \otimes w_{d}))\\
&= (-1)^{\sum_{i=1}^d l_{i}} \mbox{Trace}(\sum \theta(\Box_1 \otimes (X^{l_{1}}w_{1}))\otimes \cdots \otimes \theta(\Box_d \otimes (X^{l_{d}} w_{d})))\\
&=(-1)^{\sum_{i=1}^d l_{i}} \mbox{Trace}( b \circ Y').
\end{align*}
\end{proof}

\begin{lemma}\label{lemma25}
For any $F\in[\End(V^{\otimes d})^*]^G$,
there exists a $b_{F} \in B_{d}$ such that
\begin{equation}\label{l=0lemma}
F(X_{1}\otimes X_{2}\otimes\cdots\otimes X_{d})=\mbox{Trace}(b_{F} \circ X_{1}\otimes X_{2}\otimes \cdots \otimes X_{d}).
\end{equation}
\end{lemma}

\begin{proof}
Define a linear map
$J:B_{d} \longrightarrow \End(V^{\otimes d})^*$ by
$$J(b)(X_{1}\otimes X_{2} \cdots \otimes X_{d})=\mbox{Trace}(b \circ X_{1}\otimes X_{2} \cdots \otimes X_{d}).$$

For any $g\in G$, we check that
\begin{align*}
(g\cdot J(b))(X_{1}\otimes X_{2} \cdots \otimes X_{d})
&=J(b)(g^{-1}\cdot(X_{1}\otimes X_{2} \cdots \otimes X_{d}))\\
&=\mbox{Trace}(b \circ g \circ X_{1}\otimes X_{2} \cdots \otimes X_{d} \circ g^{-1})\\
&=\mbox{Trace}(g \circ b \circ X_{1}\otimes X_{2} \cdots \otimes X_{d} \circ g^{-1})\\
&=\mbox{Trace}(b \circ X_{1}\otimes X_{2} \cdots \otimes X_{d})\\
&=J(b)(X_{1}\otimes X_{2} \cdots \otimes X_{d}),
\end{align*}
where the third equality holds because the actions of $G$ and $B_d$ on $V^{\otimes d}$ commute with each other. Therefore $J(B_d)\in[\End(V^{\otimes d})^*]^G$.

Non-degeneracy of $\mbox{Trace}(-\circ-)$ on $\End(V)^{\otimes d}$ implies that $J$ is injective.
So $$\dim B_d=\dim\End_{G}(V^{\otimes d})=\dim[\End(V^{\otimes d})^*]^G$$
implies that $$J(B_d)=[\End(V^{\otimes d})^*]^G.$$ The lemma then follows by taking $b_F=J^{-1}(F)$.
\end{proof}

Any $F\in[\End(V^{\otimes d})^*]^G$ can be viewed as a function $\widehat{F}\in\mathbb{C}[\mbox{End}(V)^{\oplus d}]$  by
$$\widehat{F}(X_{1},X_2,\ldots, X_{d}):=F(X_{1}\otimes X_{2}\cdots\otimes X_{d}).$$
Thanks to Theorem \ref{mith} and the fact that $F$ is linear in variables $X_{1},\ldots, X_{d}$, we know that $F$ should be a sum of functions in terms of
$$\mbox{Trace}(U_{j_{1}} \cdots U_{j_{s}})\mbox{Trace}(U_{j_{s+1}} \cdots U_{j_{k}})\cdots \mbox{Trace}(U_{j_{t+1}} \cdots U_{j_{d}}),$$
where $U_{j_{i}}=X_{j_{i}}$ or $X_{j_{i}}^{\iota}$, and $(j_1,\ldots,j_s,j_{s+1},\ldots,j_k,\ldots,j_{t+1},\ldots,j_d)$ is an arrangement of $\{1,2,\ldots,d \}$.

\begin{lemma}\label{gele1}
Assume $$F(X_1\otimes\cdots\otimes X_d)=\mbox{Trace}(U_{j_{1}} \cdots U_{j_{s}})\cdots \mbox{Trace}(U_{j_{t+1}} \cdots U_{j_{d}})\in[\End(V^{\otimes d})^*]^G$$
where $U_{j_{i}}=X_{j_{i}}$ or $X_{j_{i}}^{\iota}$, and $(j_1,\ldots,j_s,\ldots,j_{t+1},\ldots,j_d)$ is an arrangement of $\{1,2,\ldots,d \}$.
Let $\textbf{l}_{1}=(l^{(1)}_{1}, \ldots, l^{(1)}_{d}), \textbf{l}_{2}=(l^{(2)}_{1}, \ldots, l^{(2)}_{d})\in\mathbb{Z}_{\geq 0}^{d}$ such that
$l^{(1)}_i=\left\{
\begin{array}{ll}
l_i,& (\mbox{if } U_i=X_i);\\
0,& (\mbox{if } U_i=X^\iota_i)
\end{array}\right.$
and
$l^{(2)}_i=\left\{
\begin{array}{ll}
l_i,& (\mbox{if } U_i=X^\iota_i);\\
0,& (\mbox{if } U_i=X_i).
\end{array}\right.$
Let $b_{F}\in B_d$ be the element determined by $F$ as in Lemma \ref{lemma25}.
For any $X\in\mathfrak{g}$, we have
\begin{align*}
&\mbox{Trace}(X(\mathbf{l}_{2}) \circ b_{F} \circ X(\mathbf{l}_{1}) \circ X_{1} \otimes \dots \otimes X_{d})=\\
&(-1)^{\sum_{i=1}^{d}l_i^{(2)}}\mbox{Trace}(X^{l_{j_{1}}}U_{j_{1}} \cdots X^{l_{j_{s}}} U_{j_{s}})\cdots
\mbox{Trace}(X^{l_{j_{t+1}}}U_{j_{t+1}} \cdots X^{l_{j_{d}}}U_{j_{d}}).
\end{align*}

\end{lemma}

\begin{proof}
Without loss of generality, it is enough to prove the lemma for $X_{i}=\theta(u_{i}\otimes w_{i})\in\End(V)$, i.e. $X_i(v)=\langle w_i,v\rangle u_i$ for any $v\in V$.
Set $$Y=X_1\otimes X_2\otimes \cdots\otimes X_d=\theta(u_{1}\otimes w_{1})\otimes \cdots \otimes \theta(u_{d}\otimes w_{d})\in\mbox{\End}(V^{\otimes d}).$$

For any $\mathbf{k}=(k_{1},\ldots, k_{d}) \in  \mathbb{Z}_{\geq 0}^{d}$, it is clear that
$$X(\mathbf{k})\circ Y=\theta(X^{k_1}u_{1}\otimes w_{1})\otimes \cdots \otimes \theta(X^{k_d}u_{d}\otimes w_{d})\in\mbox{\End}(V^{\otimes d}).$$
Therefore by Lemma \ref{change X to X'} we have
\begin{equation*}
\mbox{Trace}(X(\mathbf{l}_{2}) \circ b_F \circ X(\mathbf{l}_{1})\circ Y)=(-1)^{\sum_{i=1}^{d}l_i^{(2)}}\mbox{Trace}(b_F \circ Y')
\end{equation*}
where $Y'=\theta(u'_{1}\otimes w'_{1})\otimes \cdots \otimes \theta(u'_{d}\otimes w'_{d})$ with $u'_{i}=X^{l_i}u_i,w'_{i}=w_{i}$ if $U_{i}=X_{i}$,
and $u'_{i}=u_i,w'_{i}=X^{l_i}w_{i}$ if $U_{i}=X^{\iota}_{i}$.

Note that
\begin{align*}
&\mbox{Trace}(b_{F} \circ Y)\\
&=\mbox{Trace}(U_{j_{1}} \cdots U_{j_{s}})\cdots \mbox{Trace}(U_{j_{t+1}} \cdots U_{j_{d}})\quad\quad\quad (\mbox{by Lemma } \ref{lemma25})\\
&=\langle\omega_{j_{1}},\overline{\omega}_{j_{2}}\rangle\langle\omega_{j_{2}},\overline{\omega}_{j_{3}}\rangle\cdots\langle\omega_{j_{s}},\overline{\omega}_{j_{1}}\rangle \cdots\cdots\langle\omega_{j_{t+1}},\overline{\omega}_{j_{t+2}}\rangle\langle\omega_{j_{t+2}},\overline{\omega}_{j_{t+3}}\rangle\cdots \langle\omega_{j_{d}},\overline{\omega}_{j_{t+1}}\rangle \\
&\quad\quad\quad\quad\quad\quad\quad\quad\quad\quad\quad\quad\quad\quad\quad\quad\quad\quad\quad\quad\quad\quad
(\mbox{by Lemma }\ref{Tr of product})
\end{align*}
where $\omega_{i}=w_{i}$, $\overline{\omega_{i}}=u_{i}$ if $U_{i}=X_{i}$, and $\omega_{i}=u_{i}$, $\overline{\omega_{i}}=w_{i}$ if $U_{i}=X_{i}^\iota$.

Replacing $Y$ by $Y'$ in the above formula, we see that each $\omega_{i}$ unchanges while each $\overline{\omega}_{i}$ is replaced by $X^{l_{i}}\overline{\omega}_{i}$. Finally by using Lemma \ref{Tr of product} again, we get that
\begin{align*}
&\mbox{Trace}(X(\mathbf{l}_{2}) \circ b_F \circ X(\mathbf{l}_{1})\circ Y)=(-1)^{\sum_{i=1}^{d}l_i^{(2)}}\mbox{Trace}(b_F \circ Y')\\
&=(-1)^{\sum_{i=1}^{d}l_i^{(2)}}\mbox{Trace}(X^{l_{j_{1}}}U_{j_{1}} \cdots X^{l_{j_{s}}} U_{j_{s}})\cdots
\mbox{Trace}(X^{l_{j_{t+1}}}U_{j_{t+1}} \cdots X^{l_{j_{d}}}U_{j_{d}}).
\end{align*}
\end{proof}

\subsection{Morphisms of affine varieties}
Let $\mathbb{A}$ be an affine variety, $G$ be a reductive group acting on $\mathbb{A}$, and $\mathbb{M}$ be a linear representation of $G$. An affine variety  $\mathbb{A}$ is said to be \emph{normal} if the regular function ring $\mathbb{C}[\mathbb{A}]$ is integrally closed.

The following two lemmas were given in \cite{KP1}.
\begin{lemma}\label{exle}
For any $G$-stable closed subvariety $\mathbb{W} \subseteq \mathbb{A}$ and $G$-equivariant morphism $\psi : \mathbb{W}\rightarrow \mathbb{M}$,
there exists a $G$-equivariant morphism $\Psi: \mathbb{A} \rightarrow \mathbb{M}$ extending $\psi$.
\end{lemma}

\begin{lemma} \label{asle}
If $e\in\mathbb{A}$ satisfies that
\begin{itemize}
\item[(1)]$\overline{G\cdot e}\subset \mathbb{A}$ is normal; and
\item[(2)] $\dim(\overline{G\cdot e}\setminus G\cdot e)\leq\dim(G\cdot e)-2$,
\end{itemize}
then for any $m\in \mathbb{M}^{G_{e}}$, there exists a $G$-equivariant morphism $\Psi:\mathbb{A} \rightarrow \mathbb{M}$ such that $\Psi(e)=m$.
\end{lemma}

Specify $G=O(V)$ or $SP(V)$, $\mathbb{A}=\mbox{Lie}(G)=\mathfrak{g}\subset\mbox{End}(V)$ and $\mathbb{M}=\End(V^{\otimes d})$.

\begin{remark}
For any nilpotent element $e\in\mathfrak{g}$, the second condition in Lemma \ref{asle} always holds (c.f. Lemma 8.4 \cite{Jan2}).
\end{remark}

\subsection{$G$-equivariant morphisms $\mbox{Mor}_{G}(\mathfrak{g},\End(V^{\otimes d}))$}
Let $R:=\mathbb{C}[\mathfrak{g}]^{G}$ and denote by $\mbox{Mor}_{G}(\mathfrak{g},\mbox{End}(V^{\otimes d}))$ the set of all $G$-equivariant morphism (of varieties) from $\mathfrak{g}$ to $\mbox{End}(V^{\otimes d})$. There is an $R$-module structure on $\mbox{Mor}_{G}(\mathfrak{g},\End(V^{\otimes d}))$ given by
$$(r\circ f)(X)=r(X)f(X), \quad (\forall r\in R, f\in \mbox{Mor}_{G}(\mathfrak{g},\End(V^{\otimes d})), X\in \mathfrak{g}).$$
Let $S\subset\mbox{Mor}_{G}(\mathfrak{g},\End(V^{\otimes d}))$ be the subset consisting of those $\Psi\in  \mbox{Mor}_{G}(\mathfrak{g},\End(V^{\otimes d}))$ such that
$$\Psi(X)=(X^{l_{1}'} \otimes \dots \otimes X^{l_{d}'}) \circ b \circ (X^{l_{1}} \otimes \dots  \otimes X^{l_{d}}), \quad (\forall X\in\mathfrak{g})$$ for some $b\in B_{d}$ and $l'_1,\ldots,l'_d,l_1,\ldots,l_d\in\mathbb{Z}_{\geq 0}$.

\begin{prop}\label{kprop}
As an $R$-module, $\mbox{\em Mor}_{G}(\mathfrak{g},\End(V^{\otimes d}))$ is generated by $S$.
\end{prop}
\begin{proof}
Set $N=\mathfrak{g}\oplus \End(V)^{\oplus d}$. The embedding $R \hookrightarrow \mathbb{C}[N]^{G}$ induces an $R$-module structure on $\mathbb{C}[N]^{G}$.   Consider the $R$-module homomorphism
$$J: \mbox{Mor}_{G}(\mathfrak{g},\End(V)^{\oplus d}) \longrightarrow \mathbb{C}[N]^{G}, \quad \Psi\mapsto J(\Psi)$$
defined by $$J(\Psi)(X, X_1,X_2,\ldots,X_d)= \mbox{Trace}(\Psi(X)\circ (X_1\otimes X_2\otimes \cdots\otimes X_d))$$
for any $X\in\mathfrak{g}$ and $X_1,X_2,\ldots,X_d\in\End(V)$.
Observe that $J(\Psi)$ is linear in variables $X_1,X_2,\ldots,X_d$.

Non-degeneracy of $\mbox{Trace}(\End(V^{\otimes d})\circ\End(V^{\otimes d}))$ implies that $J$ is injective. Therefore we only need to prove $R J(S)=J (\mbox{Mor}_{G}(\mathfrak{g},\End(V^{\otimes d})))$.

\noindent\textbf{Claim:} $ J(\Psi)(X,X_{1},\ldots, X_{d})$ is in form of
\begin{equation*}
\sum r(X)\mbox{Trace}(X^{l_{1}}U_{j_{1}}X^{l_{2}}U_{j_{2}} \cdots X^{l_{s}}U_{j_{s}})\cdots \mbox{Trace}(X^{l_{t+1}}U_{j_{t+1}}X^{l_{t+2}}U_{j_{t+2}} \cdots X^{l_{d}}U_{j_{d}})
\end{equation*}
where $l_i\in\mathbb{Z}_{\geq 0}$, $U_{i}=X_{i}$ or $X^{\iota}_{i}$, $(j_1,\ldots,j_s,\cdots, j_{t+1},\ldots,j_d)$ is an arrangement of $\{1,2,\ldots,d \}$ and $r\in R$.

\noindent\textit{Proof of the claim:} \\
Choose a $G$-equivariant extension $\Psi':\mbox{End}(V)\longrightarrow M$ of $\Psi$ by Lemma \ref{exle}. Then $J(\Psi')$ can be viewed as a $G$-invariant function on $\mbox{End}(V)^{\oplus(d+1)}$. Thus by \eqref{iotalie} and Theorem \ref{mith}, we can see that $J(\Psi)(X,X_{1},\ldots,X_{d})$ is in form of
\begin{equation*}
\sum r(X)\mbox{Trace}(X^{l_{1}}U_{j_{1}}X^{l_{2}}U_{j_{2}} \cdots X^{l_{s}}U_{j_{s}})\cdots \mbox{Trace}(X^{l_{t+1}}U_{j_{t+1}}X^{l_{t+2}}U_{j_{t+2}} \cdots
X^{l_{k}}U_{j_{k}})
\end{equation*}
with $l_i\in\mathbb{Z}_{\geq 0}$, $U_{i}=X_{i}$ or $X^{\iota}_{i}$, $r\in R$ and $ j_{i} \in \{1,2,\ldots,d \}$ for $i=1,2,\ldots,k$. Notice that $J(\Psi)(X, X_{1},\ldots, X_{d})$ is linear in variables $X_{1},\ldots, X_d$. So we have $k=d$ and $j_{i_{1}}\neq j_{i_{2}}$ if
$i_1\neq i_{2}$. We complete the proof of the claim.

Thanks to the claim, we only need to show that
\begin{equation*}
\mbox{Trace}(X^{l_{1}}U_{j_{1}}X^{l_{2}}U_{j_{2}} \cdots X^{l_{s}}U_{j_{s}})\cdots \mbox{Trace}(X^{l_{t+1}}U_{j_{t+1}}X^{l_{t+2}}U_{j_{t+2}} \cdots X^{l_{d}}U_{j_{d}}) \in J(S),
\end{equation*}
which is obvious by Lemma \ref{gele1}.
\end{proof}

\subsection{Vust's Theorem for $\O(V)$ and $\SP(V)$}
Let $e\in \mathfrak{g}$ be a nilpotent element and recall the notation $e^{(i)}\in\End(V^{\otimes d})$ in \eqref{def:ei}. Denote by $B_{d}[e]$ the subalgebra of $\End(V^{\otimes d})$ generated by $B_{d}\cup\{e^{(i)}|1\leq i \leq d\}$. The following is a generalization of Vust's Theorem \eqref{Vust:typeA} for the cases other than type $A$.

\begin{theorem} \label{mainresult}
Let $G=O(V)$ or $SP(V)$. If a nilpotent element $e\in\mathfrak{g}=\mbox{Lie}(G)$ satisfies that the nilpotent orbit closure $\overline{G\cdot e}$ is normal, then
\begin{equation}
\mbox{\em End}_{G_{e}}(V^{\otimes d})=B_{d}[e].
\end{equation}
\end{theorem}
\begin{proof}
For any $m\in\End(V^{\otimes d})^{G_{e}}$, by Lemma \ref{asle} we have a $G$-equivariant morphism $\Psi: \mathfrak{g}\rightarrow (\End(V^{\otimes d}))^{G_{e}}$ such that $\Phi(e)=m$. So
$$(\End(V^{\otimes d}))^{G_{e}}=\{\Psi(e)\mid \Psi\in\mbox{Mor}_{G}(\mathfrak{g},\End(V^{\otimes d}))\}.$$

Notice that for any $\Psi\in S$,
$$\Psi(e)=(e^{l_{1}'} \otimes \dots  \otimes e^{l_{d}'}) \circ b \circ (e^{l_{1}} \otimes \dots  \otimes e^{l_{d}})\in B_d[e].$$
Hence Proposition \ref{kprop} implies that
$$(\End(V^{\otimes d}))^{G_{e}}=\{\Psi(e)\mid \Psi\in\mbox{Mor}_{G}(\mathfrak{g},\End(V^{\otimes d}))\} \subset B_{d}[e].$$

On the other hand, it can be checked directly that
$$B_{d}[e]\subset(\End(V^{\otimes d}))^{G_{e}}.$$

So we finally obtain that
$$B_{d}[e]=(\End(V^{\otimes d}))^{G_{e}}=\End_{G_{e}}(V^{\otimes d}).$$
\end{proof}

\begin{remark}
A criteria on the normality of $\overline{G\cdot e}$ for any nilpotent element $e\in\mathfrak{g}$ can be found in \cite{KP2}.
\end{remark}

\subsection{Description of $G_e$}
It can be found in Section 3 of \cite{Jan2} that
$$G_{e}=C_{e}\rtimes R_{e}$$ where $C_{e}$ is the reductive part and $R_{e}$ is the unipotent radical.
Moreover, $R_{e}$ is connected (c.f. Proposition 3.12 in \cite{Jan2}).
Suppose that $e\in\mathfrak{g}$ corresponds to a partition $[1^{r_{1}}2^{r_{2}}\cdots ]$ of $\mbox{dim}(V)$ (by Jordan blocks), then we have an isomorphism of algebraic groups (c.f. \S 3.8 in \cite{Jan2})
$$\rho_{O(V)} :  \prod_{s\geq 1; s \ \mbox{odd}}O_{r_s} \times \prod_{s \geq 1; s \ \mbox{even}}SP_{r_{s}} \rightarrow  C_{e}, \quad \mbox{if} \ G=O(V)$$
while
$$\rho_{SP(V)} :   \prod_{s \geq 1; s \ \mbox{even}}O_{r_{s}} \times \prod_{s \geq 1; s \ \mbox{odd}}SP_{r_{s}} \rightarrow C_{e}  , \quad \mbox{if} \  G=SP(V).$$

We only describe the isomorphism $\rho_{O(V)}$. Choose  $v_1,v_2, \ldots v_r \in V$ such that $e^{d_i}v_i=0$ and $\{e^j\cdot v_{i} \mid 0\leq j \leq d_i-1,1\leq i \leq r\}$ forms a
basis of $V$. Here each number $d_i$ corresponds to the order of a Jordan block. Set
$$W_s=\sum_{i;d_i=s}\mathbb{C}v_{i}.$$
The orthogonal group $O_{r_{s}}$ is defined on $W_{s}$ by a non-degenerate symmetric bilinear form. For any $g \in O_{r_{s}}$, its image under $\rho_{O(V)}$ is given by
$$\rho_{O(V)}(g)(e^j\cdot v_{i})=\left\{
\begin{array}{ll}
e^j\cdot g v_{i}, & \mbox{if $d_i=s$};\\
e^j\cdot v_{i}, & \mbox{otherwise.}
\end{array}
\right.
$$
Therefore as an $\O_{r_s}$-module,
\begin{equation*}
V\simeq W_{s}^{\oplus s}\oplus W_{s}'
\end{equation*}
where $W_{s}$ is the standard $\O_{r_{s}}$-module and $\O_{r_s}$ acts on $W_{s}'$ trivially.

Furthermore, the above construction shows that $W_{s_1}\subset W_{s_2}'$ for any $s_1\neq s_2$.

\subsection{Vust's Theorem for $\mathfrak{so}(V)$ and $\mathfrak{sp}(V)$}
The following lemma comparing $[V^{\otimes k}]^{\SO(V)}$ and $[V^{\otimes k}]^{\O(V)}$ will be used in the proof of Theorem \ref{universal coro}.
\begin{lemma} \label{tensor invar of SO}
(1). If $\dim(V)$ is odd,  then we have $[V^{\otimes k}]^{\SO(V)}=[V^{\otimes k}]^{\O(V)}$ for all $k \in \mathbb{N}$.\\
(2). If $\dim(V)$ is even,  then we have $[V^{\otimes k}]^{\SO(V)}=[V^{\otimes k}]^{\O(V)}$ for all $k<\dim(V)$.
\end{lemma}
\begin{proof}
Statement(1) follows  from the fact $\O(V)=\SO(V)\cup (-1)\SO(V)$.

Suppose $\dim(V)=2r$ for some $r \in \mathbb{N}$.
If $k$ is odd, since $-\mbox{id}_V \in \SO(V)$ we have $[V^{\otimes k}]^{\SO(V)}=[V^{\otimes k}]^{\O(V)}=0$.
If $k$ is even, we identify $V^{\otimes k}$ with $\End(V^{\otimes k/2})$ similar to \eqref{isotheta}. Then Theorem 1.4 (2) in \cite{Gro} implies that $[\End(V^{\otimes k/2})]^{\SO(V)}=[\End(V^{\otimes k/2})]^{\O(V)}$.
Thus we have proved statement (2).
\end{proof}

Now we can obtain the Lie algebra version of Vust's Theorem for cases other than type $A$.
\begin{theorem} \label{universal coro}
Let $G=\O(V)$ or $\SP(V)$, and $e\in\mathfrak{g}=\Lie(G)$ be a nilpotent element with partition $[1^{r_{1}}2^{r_{2}}\cdots ]$ of $\dim(V)$ by Jordan blocks. Assume $e$ satisfies that
\begin{itemize}
\item[(1)] the nilpotent orbit closure $\overline{G\cdot e}$ is a normal variety;
\item[(2)] if $G=\O(V)$, either $r_{s}=\mbox{odd}$ or $r_{s}> 2d$ for all odd $s$; if $G=\SP(V)$, either $r_{s}=\mbox{odd}$ or $r_{s}> 2d$ for all even $s$.
\end{itemize}
Then we have
$$\End_{\textbf{U}(\mathfrak{g}_{e})}(V^{\otimes d})=B_{d}[e].$$
\end{theorem}
\begin{proof}
Here we will only prove the theorem for $G=\O(V)$ since a similar argument works for $G=\SP(V)$. Denote by $G_{e}^{\circ}$ the connected component of $G_{e}$ containing $\mbox{id}_V$.
By the relation between representation of connected algebraic group and its Lie algebra, we need to show $\End_{G_{e}^{\circ}}(V^{\otimes d})=B_{d}[e]$.

Set
\begin{eqnarray*}
\O_e&:=&\O_{r_{1}}\times \O_{r_{3}} \times \O_{r_{5}}\times\cdots ,\\
\SO_e&:=&\SO_{r_{1}}\times \SO_{r_{3}} \times \SO_{r_{5}}\times\cdots, \\
\SP_e&:=&\SO_{r_{2}}\times \SO_{r_{4}} \times \SO_{r_{6}}\times\cdots.
\end{eqnarray*}
Thus
$$G_{e}=O_{e}\rtimes (\SP_e \rtimes R_e)\quad\mbox{and}\quad G_{e}^{\circ}=\SO_{e}\rtimes (\SP_e \rtimes R_e).$$

We claim that
$$\End_{\O_{r_{s}}}(V^{\otimes d})=\End_{\SO_{r_{s}}}(V^{\otimes d})\quad \mbox{for all even $s$}.$$

Indeed we have
\begin{align*}
\End_{\O_{r_{s}}}(V^{\otimes d})
&=[\End(V^{\otimes d})]^{\O_{r_{s}}}\\
&\simeq[V^{\otimes 2d}]^{\O_{r_{s}}}
\quad\quad\quad(\mbox{by bijection $\theta^{\otimes d}: V^{\otimes 2d}\rightarrow \End(V^{\otimes d})$ similar to \eqref{isotheta}}) \\
&=[\bigoplus_{k=0}^{2d}(W_{s}^{\otimes k}\otimes W_{s}'^{\otimes (2d-k)})^{\oplus c_{k}}]^{\O(r_{s})} \\
&=\bigoplus_{k=0}^{2d}([W_{s}^{\otimes k}\otimes W_{s}'^{\otimes (2d-k)}]^{\O_{r_{s}}})^{\oplus c_{k}} \\
&=\bigoplus_{k=0}^{2d}([W_{s}^{\otimes k}]^{\O_{r_{s}}}\otimes W_{s}'^{\otimes (2d-k)})^{\oplus c_{k}}
\end{align*}
where $c_{k}=s^k{2d\choose k}$.
By the same procedure we have
$$\End_{\SO_{r_{s}}}(V^{\otimes d})\simeq\bigoplus_{k=0}^{2d}([W_{s}^{\otimes k}]^{\SO_{r_{s}}}\otimes W_{s}'^{\otimes (2d-k)})^{\oplus c_{k}}.$$
Therefore, the claim follows from Lemma \ref{tensor invar of SO}.

Using the above claim repeatedly, we get that
$$\End_{\O_{e}}(V^{\otimes d})=\End_{\SO_{e}}(V^{\otimes d}),$$
and hence
$$\End_{\O_{e}\rtimes (SP_e \rtimes R_e)}(V^{\otimes d})=\End_{\SO_{e}\rtimes (\SP_e \rtimes R_e)}(V^{\otimes d}).$$

Thus we obtain $\End_{G_{e}^{\circ}}(V^{\otimes d})=\End_{G_{e}}(V^{\otimes d})=B_{d}[e]$.
\end{proof}

\subsection{Double centralizer property}
Denote by $\phi$ the action of $\textbf{U}(\mathfrak{g})$ on $V^{\otimes d}$.
Though we do not give a double centralizer property for $\textbf{U}(\mathfrak{g}_{e})$ and $B_d[e]$, instead we have the following proposition.
\begin{prop}
Let $\mathfrak{g}=\mathfrak{sp}(V)$ or $\mathfrak{so}(V)$ be a simple Lie algebra of type $B$ or $C$. If the nilpotent element $e\in \mathfrak{g}$ satisfies the assumption in Theorem \ref{universal coro}. Then the following double centralizer property holds:
\begin{eqnarray}
\End_{\phi(\textbf{U}(\mathfrak{gl}(V)_{e}))\cap\phi(\textbf{U}(\mathfrak{g}))}(V^{\otimes d})&=&B_{d}[e],\label{eq:1}\\
\phi(\textbf{U}(\mathfrak{gl}(V)_{e})) \cap \phi(\textbf{U}(\mathfrak{g}))&=&\End_{B_{d}[e]}(V^{\otimes d})\label{eq:2}.
\end{eqnarray}
\end{prop}
\begin{proof}

Firstly, it is clear that actions of $B_{d}[e]$ and $\phi(\textbf{U}(\mathfrak{gl}(V)_{e})) \cap \phi(\textbf{U}(\mathfrak{g}))$ commute with each other. Thus Equation \eqref{eq:1} follows from Theorem \ref{universal coro} and the fact that $\phi(\textbf{U}(\mathfrak{gl}(V)_{e}))\cap \phi(\textbf{U}(\mathfrak{g}))\supseteq\phi(\textbf{U}(\mathfrak{g}_{e}))$.

The following duality can be found in Theorem 2.4 in \cite{BK2}:
\begin{eqnarray*}
\End_{\mathbf{U}(\mathfrak{gl}(V)_e)}(V^{\otimes d})&=&\mathfrak{S}_{d}[e];\\
\phi(\mathbf{U}(\mathfrak{gl}(V)_e))&=&\End_{\mathfrak{S}_{d}[e]}(V^{\otimes d}).
\end{eqnarray*}
Note that $B_{d}[e]\supset \mathfrak{S}_{d}[e]$ and $B_{d}[e]\supset B_d$. Thus
$$\phi(\textbf{U}(\mathfrak{gl}(V)_{e})) \cap \phi(\textbf{U}(\mathfrak{g}))\subset\End_{B_{d}[e]}(V^{\otimes d})\subset\End_{\mathfrak{S}_{d}[e]}(V^{\otimes d})=\phi(\mathbf{U}(\mathfrak{gl}(V)_e))$$
and
$$\phi(\textbf{U}(\mathfrak{gl}(V)_{e})) \cap \phi(\textbf{U}(\mathfrak{g}))\subset\End_{B_{d}[e]}(V^{\otimes d})\subset\End_{B_d}(V^{\otimes d})=\phi(\mathbf{U}(\mathfrak{g})).$$
Therefore there comes Equation \eqref{eq:2}.
\end{proof}

\begin{remark}
It is natural to ask whether
\begin{equation*}
\phi_{d}(\textbf{U}(\mathfrak{gl}_{e})) \cap \phi_{d}(\textbf{U}(\mathfrak{g}))=\phi_{d}(\textbf{U}(\mathfrak{g}_{e})).
\end{equation*}
Though we can not answer this question in general, a direct calculation shows that the above equality holds when $d=2$ and $\mbox{rank}(\mathfrak{g})\leq 3$.
\end{remark}


\section{Centralizer of W-algebra action on $V^{\otimes d}$ } \label{w-algebra}
In this section, take $\mathfrak{g}=\mathfrak{so}_{2r}, \mathfrak{so}_{2r+1}$ or $\mathfrak{sp}_{2r}$. For convenience, entries of matrices in $\mathfrak{g}$ are indexed by $I\times I$ where
$$I=\begin{cases}
\{-r,\ldots,-1,0,1,\ldots,r\}\quad\mbox{if $\mathfrak{g}=\mathfrak{so}_{2r+1}$};\\
\{-r,\ldots,-1,1,\ldots,r\}\quad\mbox{if $\mathfrak{g}=\mathfrak{so}_{2r}$ or $\mathfrak{sp}_{2r}$}.
\end{cases}$$

\subsection{Gradings}\label{presec3}
Assume that $\Gamma:\mathfrak{g}=\bigoplus_{i\in \mathbb{Z}} \mathfrak{g}(i)$ is a $\mathbb{Z}$-grading of $\mathfrak{g}$.
We say $\Gamma$ is \textit{good} for nilpotent $e \in \mathfrak{g}$ if it satisfies that
\begin{itemize}
\item[(1)] $e \in \mathfrak{g}_{2}$;
\item[(2)]$\mbox{ad}_e:\mathfrak{g}_{j}\rightarrow \mathfrak{g}_{j+2}$ is injective for $j \leq -1$; and
\item[(3)]$\mbox{ad}_e:\mathfrak{g}_{j}\rightarrow \mathfrak{g}_{j+2}$ is surjective for $j \geq -1$.
\end{itemize}
We call $\Gamma$ is \textit{even} if $\mathfrak{g}_{j}=0$ for all odd $j$.

Refer to the literature \cite{EK} for classification of nilpotent elements which admit even good gradings for classical Lie algebras.
In this section we always assume that $e$ admits an even good grading.
Moreover, an even good grading $\Gamma$ induces a grading for $\textbf{U}(\mathfrak{g})$, which is called a \emph{loop grading}.

For any $\mathbb{Z}$-grading $\Gamma$, there exists a semisimple element $h_{\Gamma} \in \mathfrak{g}$ such that $\Gamma$ coincides with the eigenspace decomposition of $\mbox{ad}_{h_{\Gamma}}$ (c.f. \cite{Wa}), i.e. $$\mathfrak{g}_j=\{x\in\mathfrak{g}|[h_{\Gamma},x]=jx\}.$$
Let $\mathfrak{h}$ be a Cartan subalgebra of $\mathfrak{g}$ containing $h_\Gamma$.

\vspace{0.3cm}
\noindent{\bf Convention:}
\emph{Without loss of generality, we assume that $h_{\Gamma}$ is diagonal (by a conjugate transformation if necessary), and hence we take $\mathfrak{h}$ to be the standard Cartan subalgebra consisting of all diagonal matrices in $\mathfrak{g}$.}
\vspace{0.3cm}

Write
$F_{i,j}:=E_{i,j}-\theta_{i,j}E_{-j,-i} (i,j\in I)$ with
$$\theta_{i,j}=\left\{\begin{array}{cl}
1, & \mbox{if $\mathfrak{g}=\mathfrak{so}_{2r+1}$ or $\mathfrak{so}_{2r}$};\\
\mbox{sgn}(i)\mbox{sgn}(j), & \mbox{if $\mathfrak{g}=\mathfrak{sp}_{2r}$}.\\
\end{array}\right.
$$
The following set
$$\mathbb{B}= \begin{cases}
\{F_{i,i} \mid 0<i\leq r\}\cup\{F_{\pm i,\pm j} \mid 0<i<j\leq r\}\cup\{F_{0,\pm i}\mid 0<i\leq r\}, \  \mbox{if $\mathfrak{g}=\mathfrak{so}_{2r+1}$};\\
\{F_{\pm i,\pm j} \mid 0<i<j\leq r\}\cup\{F_{i,i},F_{-i,i},F_{i,-i} \mid 0 < i\leq r\},\  \mbox{if $\mathfrak{g}=\mathfrak{sp}_{2r}$};\\
\{F_{\pm i,\pm j} \mid 0<i<j\leq r\}\cup\{F_{i,i}\mid 0<i\leq r\},\ \mbox{if $\mathfrak{g}=\mathfrak{so}_{2r}$},
\end{cases}$$ forms a basis of $\mathfrak{g}$.
The subset $\{F_{i,i}=E_{i,i}-E_{-i,-i}|0<i\leq r\}\subset\mathbb{B}$ forms a basis of $\mathfrak{h}$.

Define a map
$$\mbox{col}: I\rightarrow\mathbb{Z},\ i\mapsto \mbox{col}(i)\quad \mbox{such that}\quad h_{\Gamma}\cdot v_{i}=\mbox{col}(i)v_{i}.$$

Equip $V$ a $\mathbb{Z}$-grading by $\mbox{gr}(v_{i}):=\mbox{col}(i)$. It is easy to check that $V$ is a graded $\mathfrak{g}$-module under this grading.

The set $$\{v_{i_{1}} \otimes \dots \otimes v_{i_{d}}\mid (i_{1},\dots ,i_{d})\in I^d\}$$
forms a homogeneous basis of graded $\mathfrak{g}$-module $V^{\otimes d}$ with
\begin{equation}\label{gradingofvd}
\mbox{gr}(v_{i_{1}} \otimes \dots \otimes v_{i_{d}})=\sum_{k=1}^{d}\mbox{col}(i_{k}).
\end{equation}

Set $\mathfrak{p}=\bigoplus_{i\geq 0}\mathfrak{g}(i)$ and $\mathfrak{m}=\bigoplus_{i<0}\mathfrak{g}(i)$.
The map $\mbox{col}$ satisfies the following proposition.
\begin{prop} \label{parabolic de prop }
\begin{itemize}
\item [(1)] $\mbox{col}(i)+\mbox{col}(-i)=0$, $(\forall 1\leq i \leq r)$;
\item [(2)] $F_{i,j}\in \mathfrak{p} \Leftrightarrow \mbox{col}(j)\leq \mbox{col}(i)$, $(\forall 1\leq i,j\leq r)$;
\item [(3)] $F_{i,j}\in \mathfrak{m} \Leftrightarrow \mbox{col}(j)>\mbox{col}(i)$, $(\forall 1\leq i,j\leq r)$.
\end{itemize}
\end{prop}
\begin{proof}
Assume $h_{\Gamma}=\sum_{1\leq i \leq r}a_{i}(E_{i,i}-E_{-i,-i})$. It is clear that $$\mbox{col}(i)=a_i,\quad \mbox{col}(-i)=-a_i \quad (\forall 1\leq i\leq r)$$
and $$\mbox{gr}(F_{i,j})=\mbox{col}(i)-\mbox{col}(j).$$
So the proposition follows.
\end{proof}

\subsection{W-algebra $\mathcal{W}_{\chi}$}
There are several equivalent definitions for W-algebras. Here we adapt the following definition for those nilpotent element $e\in\mathfrak{g}$ admitting an even good grading.

Let $\chi\in\mathfrak{g}^{*}$ be the linear function on $\mathfrak{g}$ uniquely determined by
$$\chi(g)=\mbox{Trace}(\mbox{ad}_e\circ \mbox{ad}_X), \quad(\forall X\in\mathfrak{g}).$$
Let $I_{\chi}$ be the left ideal of $\textbf{U}(\mathfrak{m})$ generated by $a-\chi(a)$ for all $a\in\mathfrak{m}$.

The \textit{W-algebra associated to $e$} is defined as
$$\mathcal{W}_{\chi}:=\{ y \in \textbf{U}(\mathfrak{p}) \mid [a,y] \in I_{\chi}, \forall a \in \mathfrak{m}\}.$$
The W-algebra $\mathcal{W}_{\chi}$ is a filtration subalgebra of graded algebra $\textbf{U}(\mathfrak{g})$ (with loop grading). By restriction, $V^{\otimes d}$ has a $\mathcal{W}_{\chi}$-module structure which is compatible with the above filtration of $\mathcal{W}_{\chi}$.

It is clear by the definition of good grading that $\mathfrak{g}_e\subset\mathfrak{p}$. So there is an embedding $\textbf{U}(\mathfrak{g}_e)\hookrightarrow \textbf{U}(\mathfrak{p})$.

\begin{theorem}[c.f. Theorem 3.8 in \cite{BGK}]\label{pr's loop grading lemma}
The embedding $\textbf{U}(\mathfrak{g}_e)\hookrightarrow \textbf{U}(\mathfrak{p})$ induces a graded algebra isomorphism
$$\textbf{U}(\mathfrak{g}_e)\simeq \mbox{gr}(\mathcal{W}_\chi).$$
\end{theorem}


\subsection{Tensor identities} All statements in this subsection can be found in \cite{BK1, BK3}. Though Brundan and Kleshchev dealt with case of type $A$ only, their proofs are still valid for types $B,C$ and $D$ when $e$ admits an even good grading.

Set the quotient space
$$Q_{\chi}:=\textbf{U}(\mathfrak{g})/I_{\chi}.$$
Denote by $1_{\chi}$ the coset of $1\in\textbf{U}(\mathfrak{g})$ in $Q_{\chi}$.
The vector space $Q_{\chi}$ possesses a $(\textbf{U}(\mathfrak{g}),\mathcal{W}_{\chi})$-bimodule structure, where the left action of $\textbf{U}(\mathfrak{g})$ is given by
$$u\circ u'1_{\chi}=(uu')1_{\chi}\quad (\forall u,u'\in\textbf{U}(\mathfrak{g}))$$
while the right action of $\mathcal{W}_{\chi}$ is given by
$$(u'1_{\chi})w=(u'w)1_{\chi}\quad (\forall w\in\mathcal{W}_{\chi},u'\in\textbf{U}(\mathfrak{g})).$$
We also have an isomorphism $\mathcal{W}_{\chi}\rightarrow \End_{\textbf{U}(\mathfrak{g})}(Q_{\chi})$. It has been known (c.f. \cite{BK3}) that $Q_{\chi}$ is a free $\mathcal{W}_{\chi}$-module and there exist $a_{1},\ldots,a_{h}\in \mathfrak{p}$ such that $\{ a_{1}^{i_{1}} \cdots a_{h}^{i_h}1_{\chi} \mid i_{1}, \ldots, i_{h} \geq 0 \}$ forms a basis of $Q_\chi$ as a free $\mathcal{W}_{\chi}$-module.

Denote by $\mathcal{C}(\chi)$ the
category consisting of all $\mathfrak{g}$-modules on which $a-\chi(a)$ acts locally nilpotently for all $a\in\mathfrak{m}$.
Skrybian's equivalence theorem says that the functor
\begin{eqnarray*}
Q_{\chi}\otimes_{\mathcal{W}_{\chi}}?: \mathcal{W}_{\chi}\mbox{-mod} &\rightarrow& \mathcal{C}(\chi),\\
M &\mapsto& Q_{\chi}\otimes _{\mathcal{W}_{\chi}}M
\end{eqnarray*} is an equivalence of categories.

Given $M\in\mathcal{C}(\chi)$, the subspace
$$\mbox{Wh}(M):=\{v\in M\mid xv=\chi(x)v, \forall x\in\mathfrak{m}\}$$
has a natural $\mathcal{W}_{\chi}$-module structure.  Thus we have a functor
\begin{eqnarray*}
\mbox{Wh}: \mathcal{C}(\chi)&\rightarrow& \mathcal{W}_{\chi}\mbox{-mod},\\
 M &\mapsto& \mbox{Wh}(M),
\end{eqnarray*}
which is the inverse of $Q_{\chi}\otimes_{\mathcal{W}_{\chi}}?$.

Let $W$ be an arbitrary finite dimensional $\mathfrak{g}$-module. Suppose that $W$ has a basis $\{w_{1}, \ldots, w_{r}\}$.
Define a functor
\begin{eqnarray*}
?\circledast W :\mathcal{W}_{\chi}\mbox{-mod}&\rightarrow&\mathcal{W}_{\chi}\mbox{-mod},\\
M &\mapsto& M\circledast W :=\mbox{Wh}((Q_{\chi}\otimes_{\mathcal{W}_{\chi}} M)\otimes W).
\end{eqnarray*}
Define $c_{i,j} \in \textbf{U}(\mathfrak{g})^{*}$ via the equation
$$uw_{j}=\sum_{i=1}^{r}c_{i,j}(u)w_{i} \quad \mbox{for any $u\in\textbf{U}(\mathfrak{g})$}.$$

Take a projection $\p:Q_{\chi} \twoheadrightarrow \mathcal{W}_{\chi}$ with $\p(1_{\chi})=1$.
Define a linear map of vector space by
$$\chi_{M,W}:M\circledast W \rightarrow M \otimes W,\quad (u1_{\chi} \otimes m) \otimes w\mapsto \p(u1_{\chi})m\otimes w.$$

\begin{theorem}[c.f Theorem 8.1 in \cite{BK1}] \label{tensor identy thm}
For any left $\mathcal{W}_{\chi}$-module $M$ and finite dimensional $\mathfrak{g}$-module $W$, the linear map $\chi_{M,W}$ is an isomorphism of vector
space and
$$\chi_{M,W}^{-1}(m\otimes w_j)= \sum^{r}_{i=1}(x_{i,j}\cdot 1_{\chi}\otimes m)\otimes w_i,$$
Where $(x_{i,j})_{1 \leq i,j \leq r}$ is a matrix with entries in $\textbf{U}(\mathfrak{p})$ determined uniquely by the properties
\begin{itemize}
\item[(1)] $\p(x_{i,j}1_{\chi})=\delta_{i,j}$; and

\item[(2)] $[a,x_{i,j}]+\sum_{s=1}^{r}c_{i,s}(a)x_{s,j} \in \textbf{U}(\mathfrak{g})I_{\chi}$ for any $a\in \mathfrak{m}$.
\end{itemize}
\end{theorem}

Any $\textbf{U}(\mathfrak{p})$-module $M$ can be viewed as a $\mathcal{W}_{\chi}$ module by restriction. For any $\mathfrak{g}$-module $W$,
define a linear map
$$\mu_{M,W}:M\circledast W \rightarrow M\otimes W,\quad (u1_{\chi} \otimes m) \otimes w\mapsto um\otimes w.$$
for all $u \in \mathfrak{p}$, $m \in M$ and $w\in W$.

\begin{corollary}[c.f Corollary 8.2 in \cite{BK1}]\label{tensor identy coro}
For any $\textbf{U}(\mathfrak{p})$-module $M$ and finite dimensional $\mathfrak{g}$-module $W$, $\mu_{M,W}$ is an isomorphism of $\mathcal{W}_{\chi}$-modules and
$$\mu_{M,V}^{-1}(m\otimes v_k)= \sum^{r}_{i,j=1}(x_{i,j}\cdot 1_{\chi}\otimes y_{j,k}m)\otimes v_i$$
where $(x_{i,j})_{1 \leq i,j \leq r}$ is the matrix defined in Theorem \ref{tensor identy thm} and $(y_{i,j})_{1 \leq i,j \leq r}$ is its inverse matrix.
\end{corollary}

\begin{theorem}[c.f Lemma 3.2 in \cite{BK3}] \label{bk1 theorem} Let $M=\mathbb{C}1_{M}$ be a one dimensional $\textbf{U}(\mathfrak{p})$-module.
There exist $x_{i,j}\in\textbf{U}(\mathfrak{p})$ $(1\leq i,j\leq r)$ such that
\begin{itemize}
\item [(1)] $[a,x_{i,j}]+\sum_{s=1}^{r}c_{i,s}(a)x_{s,j} \in \textbf{U}(\mathfrak{g})I_{\chi}$ for any $a\in \mathfrak{m}$;
\item[(2)] $x_{i,j}$acts on $M$ as the scalar $\delta_{i,j}$.
\end{itemize}
For any such choice of $x_{i,j}\in\textbf{U}(\mathfrak{p})$ $(1\leq i,j\leq r)$, we have
$$\mu_{M,V}^{-1}(1_{M} \otimes v_{j})=\sum_{i=1}^{r}x_{i,j}1_{\chi}\otimes 1_{M}\otimes v_{i}.$$
\end{theorem}

\begin{proof}
Denote by ${\bf c}$ the linear function on $\mathfrak{p}$ determined by
$$a\cdot 1_M=\textbf{c}(a)1_M \quad (\forall a\in \mathfrak{p}).$$
Specify the projection $\p$ in Theorem \ref{tensor identy thm} by $\p(a_{1}^{i_{1}} \cdots a_{h}^{i_h}1_{\chi})={\textbf{c}}(a_{1}^{i_{1}}) \cdots {\textbf{c}}(a_{h}^{i_h})$. Then the statement follows from Theorem \ref{tensor identy thm} and Corollary \ref{tensor identy coro}.
\end{proof}

\subsection{Degenerate affine braid algebra}
For any $g\in\mathbb{B}$, denote by $g^*\in\mathfrak{g}$ its dual with respect to the Killing form.
Let $\kappa=\sum_{g\in \mathbb{B}} gg^*\in\textbf{U}(\mathfrak{g})$ be the Casimir element.

\begin{defn}\label{defbmw}
\textit{Degenerate affine braid algebra} $\mathcal{B}_{d}$ is defined by generators
$\tilde{s}_1,\ldots,\tilde{s}_{d-1}$, $\tilde{\kappa}_{0},\ldots,\tilde{\kappa}_{d}$ and $\tilde{\gamma}_{i,j}$ $(0\leq i\neq j\leq d)$
with some relations (refer to Theorem 1.1 in \cite{ZA} since it occupies too much space and will not be used in this paper).
\underline{}\end{defn}

Let $V$ be the natural $\mathfrak{g}$-module with a standard basis $\{v_i|i\in I\}$, and $M$ be any $\mathfrak{g}$-module. There is an action $\tilde{\Phi}: \mathcal{B}_{d}\rightarrow \mbox{End}(M\otimes V^{\otimes d})$ as follows.
\begin{equation}\label{actionofBd}
\left\{
\begin{array}{l}
\tilde{\Phi}(\tilde{s}_i)=1^{\otimes i} \otimes P\otimes 1^{\otimes (d-1-i)},(i=1,\ldots,d);\\
\tilde{\Phi}(\tilde{\kappa}_i)=1^{\otimes i}\otimes\kappa\otimes1^{\otimes (d-i)},(i=0,\ldots,d);\\
\tilde{\Phi}(\tilde{\gamma}_{i,j})=\sum_{g\in\mathbb{B}}1^{\otimes i}\otimes g\otimes 1^{\otimes (j-i-1)}\otimes g^*\otimes 1^{\otimes (d-j)},(0\leq i<j\leq d),
\end{array}
\right.
\end{equation} where $P$ is the linear operator such that $P(u\otimes v)=v\otimes u$.
This action of $\mathcal{B}_{d}$ on $M\otimes V^{\otimes d}$ commutes with the action of $\textbf{U}(\mathfrak{g})$ (c.f. Theorem 1.2 \cite{ZA}).


\subsection{Action of $\mathcal{B}_{d}$ on $V^{\otimes d}$} \label{section action of B_d}

Let $\mathbb{C}_{e}$ be the trivial $\textbf{U}(\mathfrak{p})$-module, which can be viewed as a $\mathcal{W}_{\chi}$-module by restriction. Hence $Q_{\chi}\otimes_{\mathcal{W}_{\chi}}\mathbb{C}_{e}$
is a $\mathfrak{g}$-module due to Skrybian's equivalence theorem.
Then there is a $\mathcal{B}_{d}$ action on $(Q_{\chi}\otimes_{\mathcal{W}_{\chi}}\mathbb{C}_{e})\otimes V^{\otimes d}$ via $\tilde{\Phi}$.
The subspace $\mbox{Wh}((Q_{\chi}\otimes_{\mathcal{W}_{\chi}}\mathbb{C}_{e})\otimes V^{\otimes d})$ is
invariant under $\tilde{\Phi}(\mathcal{B}_{d})$ since the action of $a-\chi(a)$ $(\forall a\in\mathfrak{m})$ commutes with $\tilde{\Phi}(\mathcal{B}_{d})$. Thus we have an action of
$\mathcal{B}_{d}$ on
$\mbox{Wh}((Q_{\chi}\otimes_{\mathcal{W}_{\chi}}\mathbb{C}_{e})\otimes V^{\otimes d})=\mathbb{C}_{e}\circledast V^{\otimes d},$
which commutes with the action of $\mathcal{W}_{\chi}$.

Thanks to the following isomorphisms of $\mathcal{W}_{\chi}$-modules
$$\begin{array}{rcccl}
V^{\otimes d}&\simeq&\mathbb{C}_{e}\otimes V^{\otimes d}&\simeq&\mathbb{C}_{e}\circledast V^{\otimes d}\\
w&\mapsto& 1\otimes w &\mapsto& \mu_{\mathbb{C}_{e},V^{\otimes d}}^{-1}(1\otimes w),
\end{array}$$
we obtain a $\mathcal{B}_{d}$ action (denoted by $\Phi$) which commutes with the action of $\mathcal{W}_{\chi}$.

The following lemma can be obtained by a straightforward calculation.
\begin{lemma}\label{sandgamma}
We have
$$\Phi(\tilde{s}_i)=1^{\otimes (i-1)} \otimes P\otimes 1^{\otimes (d-1-i)}, \quad(1\leq i\leq d)$$
and
$$\Phi(\tilde{\gamma}_{i,j})=\sum_{g\in\mathbb{B}}1^{\otimes (i-1)}\otimes g\otimes 1^{\otimes (j-i-1)}\otimes g^*\otimes 1^{\otimes (d-j)}=-\gamma_{i,j}+s_{i,j}, \quad(0<i<j\leq d),$$
where $s_{i,j}$ is the endomorphism of $V^{\otimes d}$ permuting the $i$-th and $j$-th tensor positions.
\end{lemma}

Write $v_\textbf{i}:=v_{i_1}\otimes v_{i_2}\otimes \cdots\otimes v_{i_d}$ for any $\textbf{i}=(i_1,i_2,\ldots,i_d)\in I^d$.
\begin{lemma} \label{gr key le}
For any $ 1\leq k\leq d$ and $\textbf{\em i}\in I^d$, we have
$$\Phi(\tilde{\gamma}_{0,k})\cdot v_{\bf i}=e^{(k)}\cdot v_{\bf i}+\mbox{lower terms associated to the grading \eqref{gradingofvd}}.$$
\end{lemma}
\begin{proof} Recall $e^{(k)}$ in \eqref{def:ei}. The notation $F_{q,p}^{(k)}$ used in this proof is defined similarly. Write $\mu:=\mu_{\mathbb{C}_{e},V^{\otimes d}}$ for short.
We have
$$\Phi(\tilde{\gamma}_{0,k})\cdot v_{\textbf{i}}=\mu(\tilde{\Phi}(\tilde{\gamma}_{0,k}) \cdot \sum_{\textbf{j} \in I^d}(x_{\textbf{j,i}}1_{\chi}\otimes 1)\otimes v_{\textbf{j}}) =\sum_{F_{p,q}\in \mathbb{B},\textbf{j} \in I^d }\mu((F_{p,q}x_{\textbf{j,i}}1_{\chi}\otimes 1)\otimes (F_{p,q}^*)^{(k)} v_{\textbf{j}} )$$
where $x_{\textbf{j,i}}$ $(\forall \textbf{i},\textbf{j}\in I^d)$ are determined by theorem \ref{bk1 theorem}. The first equality comes from Theorem \ref{bk1 theorem} (3). The second one follows from the action of $\tilde{\gamma}_{0,k}$ constructed in Equation \eqref{actionofBd}.

If $\mbox{col}(q)\leq \mbox{col}(p)$, then by Proposition \ref{parabolic de prop } (2) we have $F_{p,q} \in \mathfrak{p}$. By Theorem \ref{bk1 theorem} (2)
we have $$\mu((F_{p,q}x_{\textbf{j,i}}1_{\chi}\otimes 1)\otimes(F_{p,q}^*)^{(k)} v_{\textbf{j}} )=F_{p,q}x_{\textbf{j,i}}\cdot 1 \otimes (F_{p,q}^*)^{(k)} v_{\textbf{j}}=0.$$

If $\mbox{col}(q) > \mbox{col}(p) $ then  by Proposition \ref{parabolic de prop } (2), we have $F_{p,q} \in \mathfrak{m}$. Thus Theorem \ref{bk1 theorem} (1) implies that
$$\mu((F_{p,q}x_{\textbf{j,i}}1_{\chi}\otimes 1)\otimes (F_{p,q}^*)^{(k)} v_{\textbf{j}})= \mu((x_{\textbf{j,i}} F_{p,q}1_{\chi} \otimes 1\otimes (F_{p,q}^*)^{(k)} v_{\textbf{j}} + \sum_{\textbf{s} \in I^d}c_{\textbf{j},\textbf{s}}(F_{p,q})x_{\textbf{s},\textbf{i}} 1_{\chi} \otimes 1\otimes (F_{p,q}^*)^{(k)} v_{\textbf{j}}).$$
Since $F_{p,q}1_{\chi}=\chi (F_{p,q})$, we have
\begin{equation} \label{left}
 \mu((x_{\textbf{j,i}} F_{p,q}1_{\chi} \otimes 1 )\otimes (F_{p,q}^*)^{(k)} v_{\textbf{j}})=
\left\{\begin{array}{ll}
0, &  \mbox{if $\textbf{j}\neq \textbf{i}$};\\
\chi(F_{p,q})(F_{p,q}^*)^{(k)}v_{\textbf{i}}, & \mbox{if $\textbf{j}=\textbf{i}$}
\end{array}\right.
\end{equation}
and
\begin{equation} \label{rigt}
\mu(c_{\textbf{j},\textbf{s}}(F_{p,q})x_{\textbf{s},\textbf{i}} 1_{\chi} \otimes 1\otimes (F_{p,q}^*)^{(k)} v_{\textbf{j}} )
=\begin{cases}
0, \ \  \mbox{if}  \  \textbf{s}\neq \textbf{i}  \ \mbox{or} \ c_{\textbf{j},\textbf{s}}(F_{p,q})=0, \\
 c_{\textbf{j},\textbf{i}}(F_{p,q})(F_{p,q}^*)^{(k)} v_{\textbf{j}},\ \mbox{otherwise}.
\end{cases}
\end{equation}

A direct calculation shows that
$$ F_{p,q}^*=F_{q,p} \quad \mbox{if} \quad p=q; \quad \quad F_{p,q}^*=\frac{1}{2}F_{q,p} \quad \mbox{if} \quad p=-q.$$

Finally we obtain the term $e^{{(k)}}\cdot v_{\textbf{i}}$ by summing up Equation \eqref{left} over all $F_{p,q}\in\mathbb{B}$ with $\mbox{col}(p)>\mbox{col}(q)$ and $\textbf{j} \in I^d$, while the lower terms come from summing up Equation \eqref{rigt} over all $F_{p,q}\in\mathbb{B}$ with $\mbox{col}(p)>\mbox{col}(q)$ and $\textbf{j} \in I^d$.
\end{proof}

\subsection{ Higher level Schur-Weyl duality}\label{ Proof of Theorem higr shur w dual}
Following is a half of the higher level Schur-Weyl duality for types $B$, $C$ and $D$.

\begin{theorem}\label{higr shur w dual}
Let $G=\O(V)$ or $\SP(V)$, and $e$ be a nilpotent element in $\mathfrak{g}=\mbox{Lie}(G)$ with partition $[1^{r_{1}}2^{r_{2}}\cdots ]$ of $\dim(V)$ by Jordan blocks. Assume $e$ satisfies that
\begin{itemize}
\item[(1)] the nilpotent orbit closure $\overline{G\cdot e}$ is a normal variety;
\item[(2)] if $G=\O(V)$, either $r_{s}$ is odd or $r_{s}> 2d$ for all for odd $s$; if $G=\SP(V)$, either $r_{s}$ is odd or $r_{s}> 2d$ for all for even $s$.
\item[(3)] $e$ admits an even good grading $\Gamma:\mathfrak{g}=\bigoplus_{i \in 2\mathbb{Z}}\mathfrak{g}(i)$.

\end{itemize}
Then
\begin{equation}
\mbox{\em End}_{\mathcal{W}_{\chi}}(V^{\otimes d})=\Phi(\mathcal{B}_{d}).
\end{equation}
\end{theorem}
\begin{proof} Notice that the action of $\mathcal{W}_{\chi}$ on $V^{\otimes d}$ is compatible with the filtration of $\mathcal{W}_{\chi}$.
 Hence we have an action of  $\mbox{gr}(\mathcal{W}_{\chi})$ on $V^{\otimes d}$. The canonical isomorphism $\mbox{gr}(\mathcal{W}_{\chi})\simeq U(\mathfrak{g}_{e})$ given in Theorem \ref{pr's loop grading lemma} implies that the above action of $\mbox{gr}(\mathcal{W}_{\chi})$
 coincides with the action of $U(\mathfrak{g}_{e})$ on $V^{\otimes d}$.

As a subalgebra of the graded algebra $\End(V^{\otimes d})$, $\Phi(\mathcal{B}_{d})$ admits a natural filtrated algebra structure. And hence there is a natural embedding $\mbox{gr}(\Phi(\mathcal{B}_{d})) \hookrightarrow \End(V^{\otimes d})$. Without confusion, we also denote the image of this embedding by the same notation $\mbox{gr}(\Phi(\mathcal{B}_{d}))$.
Since $ \Phi(\mathcal{B}_{d}) \subseteq \mbox{End}_{\mathcal{W}_{\chi}}(V^{\otimes d})$, we can calculate that
$$\mbox{gr}(\Phi(\mathcal{B}_{d})) \subseteq \mbox{End}_{\mbox{gr}(\mathcal{W}_{\chi})}(V^{\otimes d})=\mbox{End}_{U(\mathfrak{g}_{e})}(V^{\otimes d})=B_{d}[e].$$

On the other hand,
Lemmas \ref{sandgamma} and \ref{gr key le} show that $B_{d}[e] \subseteq \mbox{gr}(\Phi(\mathcal{B}_{d}))$.

So $\mbox{gr}(\Phi(\mathcal{B}_{d}))= \mbox{End}_{\mbox{gr}(\mathcal{W}_{\chi})}(V^{\otimes d})\supseteq\mbox{gr}(\mbox{End}_{\mathcal{W}_{\chi}}(V^{\otimes d}))$, which together with the fact  $\mbox{End}_{\mathcal{W}_{\chi}}(V^{\otimes d}) \supseteq \Phi(\mathcal{B}_{d})$ implies $\mbox{End}_{\mathcal{W}_{\chi}}(V^{\otimes d})=\Phi(\mathcal{B}_{d})$.
\end{proof}

\section*{Acknowledgements} We thank Weiqiang Wang, Bin Shu, Linliang Song and Yang Zeng for their helpful discussions. The first author is supported by NSF of China (Grant No. 11271131). The second author is partially supported by the NSF of China (Grant No. 11271130) and ``East China Normal University Outstanding Doctoral Dissertation Cultivation Plan of Action'' (Grant No. PY2015039).


\begin{thebibliography}{ABC}
\bibitem[B]{B} R. Brauer, {\em On algebras which are connected with the semisimple continuous groups}, Ann. Math., \textbf{38} (1937), 854--872.
\bibitem[BGK]{BGK} J. Brundan, S. M. Goodwin and A. Kleshchev, {\em Highest weight theory for finite W-algebras}, Inter. Math. Res. Notices, \textbf{15} (2008), Art. ID rnn051, 53 pp.
\bibitem[BK1]{BK1} J. Brundan and A. Kleshchev, {\em Shifted Yangians and finite W-algebras}, Adv. Math., \textbf{200} (2006), 136--195.
\bibitem[BK2]{BK2} J. Brundan and A. Kleshchev, {\em Schur-Weyl duality for higher levels}, Selecta. Math., \textbf{14} (2008), 1--57.
\bibitem[BK3]{BK3} J. Brundan and A. Kleshchev, {\em Representations of shifted Yangians and finite W-algebras}, Mem. Amer. Math. Soc., \textbf{196} (2008), no. 918, 107 pp.
\bibitem[BW]{BW} H. Bao and W. Wang, {\em A new approach to Kazhdan-Lusztig theory of type $B$ via quantum symmetric pairs}, arXiv:1310.0103.
\bibitem[CGM]{CGM} H. Chen, N. Guay and X. Ma, {\em Twisted Yangians, twisted quantum loop algebras and affine Hecke algebras of type $BC$}, Trans. Amer. Math. Soc., \textbf{366} (2014), 2517-2574.
\bibitem[D]{D} S. Donkin, {\em The normality of closures of conjugacy classes of matrices}, Invent. Math., \textbf{101} (1990), 717--736.
\bibitem[DRV]{ZA} Z. Daugherty, A. Ram and R. Virk, {\em Affine and degenerate affine BMW algebras: Actions on tensor space}, Selecta. Math., \textbf{19} (2012), 611--653.
 \bibitem[EK]{EK} P. Elashvili and V. Kac, {\em Classification of good gradings of simple Lie algebras}, in:``Lie groups and invariant theory'',  Amer. Math. Soc. Transl., \textbf{213} (2005), 85--104.
\bibitem[Gre]{Gre} R. M. Green, {\em Hyperoctaher Schur algebras}, J. Algebra, \textbf{192} (1997), 418--438.
\bibitem[Gro]{Gro} C. Grood, {\em Brauer algebras and centralizer algebras for $\SO(2n,\mathbb{C})$}, J. Algebra, \textbf{222}(1999), 678--707.
\bibitem[GW]{GW} R. Goodman and  N. Wallach, {\em Representations and invariants of the classical groups}, Cambridge University Press, 2000.
\bibitem[Ja]{Jan2} J. C. Jantzen, {\em Nilpotent orbits in representation theory}, in ``Lie Theory (PM 228)''  Birkh$\ddot{a}$user Boston, 2004, 1--206.
\bibitem[Ji]{Ji} M. Jimbo, {\em A $q$-analogue of $U(gl(N+1))$, Hecke algebra, and the Yang-Baxter equation}, Lett. Math. Phys. \textbf{11} (1986), 247--252.
\bibitem[KP1]{KP1} H. Kraft and C. Procesi, {\em Closures of conjugacy classes of matrices are normal}, Invent. Math., \textbf{53} (1979), 227--247.
\bibitem[KP2]{KP2} H. Kraft and C. Procesi, {\em On the geometry of conjugacy classes in classical groups}, Comment. Math. Helv., \textbf{57} (1982),  539--602.
\bibitem [P]{Pro} C. Procesi, {\em The invarient thery of $n\times n$ matrices}, Adv. Math., \textbf{19} (1976), 306-381.
\bibitem[S]{S} A. Sergeev, {\em The tensor algebra of the identity representation as a module over the Lie superalgebras $\mathfrak{gl}(n,m)$ and $Q(n)$}, Math. USSR Sbornik, \textbf{51} (1985), 419--427.
\bibitem[W]{Wa} W. Wang, {\em Nilpotent orbits and finite W-algebras}, in ``Geometric representation Theory and extended affine Lie algebras'', Amer. Math. Soc., Providence, RI, 2011, 71--105.
\end{thebibliography}
\end{document}